\documentclass[a4paper]{scrartcl}
\usepackage[latin1]{inputenc}
\usepackage{amsmath}
\usepackage{bbm}
\usepackage{amssymb}
\usepackage{amsthm}
\usepackage{pdfsync}
\newtheorem{lemma}{Lemma}[section]
\newtheorem{theorem}{Theorem}[section]

\newtheorem{conjecture}{Conjecture}[section]
\newtheorem{definition}{Definition}[section]
\theoremstyle{remark}
\newtheorem{remark}{Remark}[section]

\numberwithin{equation}{section}

\begin{document}
\title{On a conjecture on sparse binomial-type polynomials by Brown, Dilcher and Manna}
\author{Wolfgang Gawronski\thanks{Department of Mathematics, University of Trier, 54286 Trier, Germany. E-mail: gawron@uni-trier.de}, Thorsten Neuschel
  \thanks{Department of Mathematics, KU Leuven, Celestijnenlaan 200B box 2400, BE-3001 Leuven, Belgium. This work is supported by KU Leuven research grant OT\slash12\slash073. E-mail: Thorsten.Neuschel@wis.kuleuven.be}}

\dedication{Dedicated to the memory of Professor F.W.J. Olver}
 
 \date{\today}
\maketitle

\paragraph{Abstract} We prove a conjecture by Brown, Dilcher and Manna on the asymptotic behavior of sparse binomial-type polynomials arising naturally in a graph theoretical context in connection with the expected number of independent sets of a graph.



\section{Introduction}
 
The sequence of lacunary polynomials given by
\begin{equation}\label{P}f_n(z)=\sum_{k=0}^n \binom{n}{k}z^{\binom{k}{2}}
\end{equation}
was recently studied in a number of papers by the authors Brown, Dilcher and Manna. Originally, this sequence of polynomials arises from the question on the expected number of independent sets of vertices of finite simple graphs, see \cite{BDM3} for more details. Moreover, those polynomials may also be considered to be interesting because of their close connection with the Jacobi Theta functions investigated in \cite{BDM4}. In the paper \cite{BDM2} the authors prove algebraic and analyic properties, and in addition to those, in \cite{BDM1} they obtain upper and lower bounds for the values of these polynomials for \(0<z<1\). Furthermore, in \cite{BDM1} the asymptotic relation
\begin{equation}\label{Log}
\log f_n\left(\tfrac{1}{y}\right)\sim \frac{1}{2\log y}\log^2 n,\quad n\rightarrow \infty,
\end{equation}
has been established. On the basis of their analytic results and numerical computations finally the authors state the

\begin{conjecture}[Brown, Dilcher, Manna \cite{BDM1}] For a fixed real number \(y>1\), as \(n\rightarrow\infty\), we have
\begin{equation}\label{R1}f_n\left(\tfrac{1}{y}\right)\sim \frac{1}{\sqrt{w(n)}}\exp\left(\frac{w(n)^2 +2 w(n)}{2\log y}\right),
\end{equation}
where \(w(n)=W(n\sqrt{y}\log y)\) and \(W\) denotes the LambertW-function.
\end{conjecture}
Here, see (\ref{Log}) and (\ref{R1}), and in the sequel the symbol \(\sim\) means that the quotient of both sides converges to unity as \(n\rightarrow \infty\). Moreover, for any real \(x>0\) the value \(W(x)\) of the LambertW-function is defined as the positive solution \(t\) of the equation \(te^t =x\) (see, e.g., \cite{NIST}, p. 111). Thus, for fixed \(y>1\), the number \(w(n)\) can be defined as the positive solution \(t\) of the equation
\begin{equation}\label{E1}te^t =n\sqrt{y}\log y.
\end{equation}
The main purpose of this note is to confirm the above conjecture by proving a slightly different form of (\ref{R1}) given by the following
\begin{theorem}\label{Main} For a fixed real number \(y>1\), as \(n\rightarrow\infty\), we have 
\begin{align}\label{R2}f_n\left(\tfrac{1}{y}\right)=&\frac{1}{\sqrt{r(n)}}\exp\left(\frac{r(n)^2 +2 r(n)}{2\log y}\right)\\\label{R3}
&\times\left(1+2\sum_{k=1}^{\infty}e^{-2 \pi^2 k^2/{\log y}} \cos(2 \pi k r(n)/{\log y})+o(1)\right),
\end{align}
where \(r(n)\) is defined as the positive solution \(t\) of the equation
\begin{equation}\label{E2}t\left(e^t+\sqrt{y}\right) =n\sqrt{y}\log y.
\end{equation}
\end{theorem}
Using the notation 
\begin{equation}\label{t1}\theta_3 (z,q)=\sum_{k=-\infty}^{\infty}q^{k^2} e^{2 i k z}=1+2\sum_{k=1}^{\infty}q^{k^2} \cos(2kz)
\end{equation}
of Jacobi's third Theta function the asymptotics in Theorem \ref{Main} can be written as 
\begin{equation}\label{t2}f_n\left(\tfrac{1}{y}\right)=\frac{1}{\sqrt{r(n)}}\exp\left(\frac{r(n)^2 +2 r(n)}{2\log y}\right)\left\{\theta_3 \left(\frac{\pi r(n)}{\log y}, e^{-2 \pi^2/{\log y}}\right)+o(1)\right\},
\end{equation}
as \(n\rightarrow \infty\).

Since \(r(n)\) grows like \(\log n\) (see (\ref{r}) below), Theorem \ref{Main} clearly implies the logarithmic asymptotics in (\ref{Log}). The reason for the consideration of equation (\ref{E2}) instead of equation (\ref{E1}) will become clear from the proof. It turns out that equation (\ref{E2}) is somewhat more natural to consider which also is supported by the numerical observations. Actually, the factor in (\ref{R2}) involving \(r(n)\) is asymptotically equivalent to the corresponding quantity in (\ref{R1}) containing \(w(n)\). Also the series in (\ref{R3}) is ``very small'' compared with the constant term \(1\) in the curly brackets (see the remarks in section 2 below). The proof of Theorem \ref{Main} relies on the asymptotic evaluation of an integral representation for \(f_n\left(1/y\right)\) (Lemma \ref{IntRep}) using a non standard version of the method of saddle points. Thereby we also think that we deal with a problem in the spirit of Frank Olver who was known for his preference of problems of this kind \cite{Olver}. As a supplement to these analytic facts, in section 3 we prove monotonicity properties for the sequence \(f_n\left(1/y\right)\).

\section{Proof of the main result}
In the first instance, we establish an integral representation for the polynomials given in (\ref{P}).
\begin{lemma} \label{IntRep}For any real number \(y>1\) we have
\begin{equation}\label{Int}f_n\left(\tfrac{1}{y}\right)=\frac{1}{\sqrt{2\pi \log y}}\int\limits_{-\infty}^{\infty}\exp\left\{-\frac{s^2}{2\log y}\right\}\left(1+\sqrt{y}e^{is}\right)^n ds.
\end{equation}
\end{lemma}
\begin{proof} We start from the following well-known result from Fourier analysis on Gaussians
\[\int\limits_{-\infty}^{\infty} \exp\left\{-\alpha s^2+2\pi i k s\right\}ds =\sqrt{\frac{\pi}{\alpha}}\exp\left\{-\frac{\pi^2 k^2}{\alpha}\right\},\quad \alpha >0,~ k\in\mathbb{Z}.\]
By a simple change of variables we obtain
\[\exp\left\{-\frac{k^2}{\alpha}\right\}=\sqrt{\frac{\alpha}{4\pi}}\int\limits_{-\infty}^{\infty} \exp\left\{-\frac{\alpha}{4} s^2+i k s\right\}ds,\quad \alpha >0,~ k\in\mathbb{Z}.\]
Using this identity for \(\alpha=2/\log y\), \(y>1\), we arrive at
\begin{equation}\label{1}\left(\frac{1}{\sqrt{y}}\right)^{k^2}=\frac{1}{\sqrt{2\pi \log y}}\int\limits_{-\infty}^{\infty} \exp\left\{-\frac{s^2}{2\log y}+i k s\right\}ds.
\end{equation}
Moreover, we observe 
\[f_n\left(\tfrac{1}{y}\right)=\sum_{k=0}^{n}\binom{n}{k}\left(\frac{1}{\sqrt{y}}\right)^{k^2} y^{k/2}.\]
Hence, using (\ref{1}), after an interchange of summation and integration and applying the binomial theorem we obtain
\[f_n\left(\tfrac{1}{y}\right)=\frac{1}{\sqrt{2\pi \log y}}\int\limits_{-\infty}^{\infty}\exp\left\{-\frac{s^2}{2\log y}\right\}\left(1+\sqrt{y}e^{is}\right)^n ds.\]
\end{proof}

Now the idea for the proof of Theorem \ref{Main} consists of an asymptotic evaluation of the integral representation (\ref{Int}) in Lemma \ref{IntRep} by following a general version of the method of saddle points as described in \cite[Ch. 5]{Bruijn}. To begin with, we present some properties of the integrand in (\ref{Int}) by studying it for saddle points. Therefore, we fix \(y>1\) and define 
\[\varphi(s)=\exp\left\{-\frac{s^2}{2\log y}\right\}\left(1+\sqrt{y}e^{is}\right)^n.\]
Computing its derivative with respect to \(s\) gives
\[\varphi'(s)=\varphi(s) \left(\frac{n i \sqrt{y} e^{is}}{1+\sqrt{y}e^{is}}-\frac{s}{\log y}\right).\]
Now, the saddle points are the zeros of the function (see \cite{Bruijn}, p. 83)
\[\frac{n i \sqrt{y} e^{is}}{1+\sqrt{y}e^{is}}-\frac{s}{\log y}\]
or, equivalently, the points \(s=it\) where \(t\) is a solution of (\ref{E2}). Obviously, there is an infinite number of such solutions \(t_k\), \(k\in\mathbb{Z}\), say, where in the first instance we only choose the number \(t_0 =r(n)\) being the unique positive solution of equation (\ref{E2}). The subsequent asymptotic analysis shows that it is appropriate to use the points \(r(n)+2\pi k i\), \(k\in\mathbb{Z}\), rather than the unknown numbers \(t_k\), if \(k\neq 0\).
Clearly, the sequence \(r(n)\) is strictly increasing and unbounded in \(n\). Furthermore, we have
\begin{equation}\label{r}r(n)\sim \log n\quad \text{and}\quad e^{r(n)}\sim\frac{n\sqrt{y}\log y}{\log n}, \quad n\rightarrow\infty.
\end{equation}
In fact, these observations can be obtained by similar aguments to those usually used to study the LambertW-function \cite{NIST}. By a standard application of Cauchy's theorem we can shift the path of integration in (\ref{Int}) from the real axis to a parallel straight line through the saddle point located on the imaginary axis at \(i r(n)\). This way, for (\ref{Int}) we obtain the representation
\begin{equation}\label{I1} f_n\left(\tfrac{1}{y}\right)=\frac{1}{\sqrt{2\pi \log y}}\int\limits_{-\infty}^{\infty} e^{\psi_n(s)} ds,
\end{equation}
where we define
\begin{equation}\label{I2}e^{\psi_n(s)}=\exp\left\{\frac{-1}{2\log y}\left(s^2+2 i r(n) s -r(n)^2\right)+n\log\left(1+\sqrt{y}e^{-r(n)}e^{is}\right)\right\}.
\end{equation}
Here and throughout the paper, \(\log z\) denotes that branch of the logarithm which is real for positive \(z\).

By construction we have \(\psi_n'(0)=0\), that is for (\ref{I1}) a saddle point is located at \(s=0\) and the absolute value of the integrand is given by
\[\vert e^{\psi_n(s)}\vert=\exp\left\{\frac{-1}{2\log y}\left(s^2 -r(n)^2\right)\right\}\left\vert1+\sqrt{y}e^{-r(n)}e^{is}\right\vert^n,\]
so that we immediately infer that there is a unique global maximum located at the saddle point. In the sequel it turns out that, as \(n\rightarrow \infty\), asymptotic contributions for the integral in (\ref{I1}) not only are given in the neighborhood of the origin, but also in the vicinity of the points \(s=2k\pi, k\in \mathbb{Z}\). To facilitate the underlying analysis for large \(n\) we consider the Taylor expansion of \(\psi_n(s)\) at the origin which we may write as 
\begin{equation}\label{Taylor}\psi_n(s)=\psi_n (0)-a(n)s^2+\sum_{\nu=3}^{\infty} b_{\nu}(n)s^{\nu}.
\end{equation}
The following explicit expressions are readily verified. Therefore we omit the straightforward computations.
\begin{lemma}\label{TaylorLem} Suppose that \(y>1\) is fixed and that \(n\) is such that \(\sqrt{y}e^{-r(n)}<1\) (see (\ref{r})). Then we have
\begin{itemize}\item[i)]\begin{equation}\label{TaylorLem1}\psi_n(0)=\frac{r(n)^2}{2\log y} +n\log \left(1+\sqrt{y}e^{-r(n)}\right),\end{equation}
\item[ii)]\begin{equation}\label{TaylorLem2}a(n)=-\frac{\psi''_n(0)}{2} =\frac{1}{2\log y}+\frac{n\sqrt{y}e^{-r(n)}}{2(1+\sqrt{y}e^{-r(n)})^2},\end{equation}
\item[iii)]\begin{equation}\label{TaylorLem3}b_{\nu}(n)=\frac{\psi_n^{(\nu)}(0)}{\nu !}=\frac{-n i^{\nu}}{\nu!}\sum_{k=1}^{\infty}k^{\nu-1}\left(-\sqrt{y}e^{-r(n)}\right)^{k},~ \nu\geq 3, \end{equation}
\item[iv)]\begin{equation}\label{TaylorLem4} e^{\psi_n(s)-\psi_n(0)}=\exp\left\{-\frac{s^2+2 i r(n) s}{2\log y}\right\}\left(\frac{1+\sqrt{y}e^{-r(n)}e^{is}}{1+\sqrt{y}e^{-r(n)}}\right)^n.\end{equation}
\end{itemize}
\end{lemma}

Here, as a matter of interest we note that the coefficients \(b_{\nu}(n)\) are connected with the Euler-Frobenius polynomials, \(P_{\nu}\) say, by the formula
\[b_{\nu}(n)=\frac{-ni^{\nu}}{\nu !}\frac{P_{\nu-1}\left(-\sqrt{y} e^{-r(n)}\right)}{\left(1+\sqrt{y}e^{-r(n)}\right)^{\nu}},\]
where the polynomials \(P_{\nu}\) can be defined by the relation (see, e.g., \cite{Comtet}, p. 245)
\[\sum_{l = 0}^{\infty} ~ l^{\nu} z^{l} = \frac{P_{\nu} (z)}{(1 - z)^{\nu+1}} ~~, ~~ \nu \in \mathbb{N}\cup\{0\}.\]

Next, for technical reasons below we consider the function 
\begin{equation}\label{h} h_n(s)=\exp\left\{-s^2 +\sum_{\nu= 3}^{\infty} b_{\nu}(n)\left(\frac{s}{\sqrt{a(n)}}\right)^{\nu}\right\}\mathbbm{1}_{[-\sqrt{a(n)},\sqrt{a(n)}]}(s),
\end{equation}
where for a set \(S\) its indicator function is denoted by \(\mathbbm{1}_S\) as customary.

\begin{lemma}\label{hLem} Suppose that \(y>1\) is fixed.
\begin{itemize}\item[i)] Then there exists an integer \(n_0(y)\), depending on \(y\) only, such that 
\begin{equation}\label{hLem1}\vert h_n(s)\vert \leq e^{-(1-\cos 1)s^2 +1}\end{equation}
for \(n\geq n_0(y)\) and all \(s\in\mathbb{R}\).
\item[ii)] For any \(s \in \mathbb{R}\) we have
\begin{equation}\label{hLem2}\lim_{n\rightarrow \infty} h_n(s)=e^{-s^2}.\end{equation}
\end{itemize} 
\end{lemma}
\begin{proof} We may suppose that \(n\) is such that \(\sqrt{y}e^{-r(n)}<1\) and \(s\in [-\sqrt{a(n)}, \sqrt{a(n)}]\). Also from (\ref{r}) and (\ref{TaylorLem2}) we deduce 
\begin{equation}\label{rr} ne^{-r(n)}\sim \frac{\log n}{\sqrt{y} \log y}, \quad a(n)\sim\frac{\log n}{2\log y},\quad n\rightarrow \infty.\end{equation}
\begin{itemize}\item[i)] In view of (\ref{h}) and using (\ref{TaylorLem3}) we conclude 
\begin{align*}\Re \sum_{\nu= 3}^{\infty} &b_{\nu}(n)\left(\frac{s}{\sqrt{a(n)}}\right)^{\nu}=\sum_{\nu= 2}^{\infty} b_{2\nu}(n)\left(\frac{s}{\sqrt{a(n)}}\right)^{2\nu}\\
&=\sum_{\nu= 2}^{\infty} \frac{(-1)^{\nu} n\sqrt{y} e^{-r(n)}}{(2\nu)!}\left(\sum_{k=1}^{\infty} k^{2\nu -1} \left(-\sqrt{y}e^{-r(n)}\right)^{k-1}\right) \left(\frac{s}{\sqrt{a(n)}}\right)^{2\nu}\\
&=n\sqrt{y} e^{-r(n)}\sum_{k=1}^{\infty} \left(-\sqrt{y}e^{-r(n)}\right)^{k-1}\frac{1}{k} \left(\cos\frac{ks}{\sqrt{a(n)}} -1+\frac{k^2 s^2}{2 a(n)}\right)\\
&=n\sqrt{y} e^{-r(n)}\left(\cos\frac{s}{\sqrt{a(n)}} -1+\frac{s^2}{2 a(n)}\right)+c_n (s),
\end{align*}
where we have
\[\vert c_n (s)\vert \leq n\left(\sqrt{y}e^{-r(n)}\right)^2 \sum_{k=2}^{\infty}\left(\sqrt{y}e^{-r(n)}\right)^{k-2}\frac{1}{k}\left(2+\frac{k^2}{2}\right)=\mathcal{O}\left(\frac{\log^2 n}{n}\right),\]
as \(n\rightarrow \infty\) (see (\ref{rr})). Here, the \(\mathcal{O}\)-constant depends on \(y\) only. 
Further, since the function \((\cos x-1)/x^2\) is increasing on the interval \([0,1]\), we obtain
\[n\sqrt{y} e^{-r(n)}\left(\cos\frac{s}{\sqrt{a(n)}} -1+\frac{s^2}{2 a(n)}\right)\leq \frac{n\sqrt{y} e^{-r(n)}}{2a(n)}s^2+\frac{n\sqrt{y} e^{-r(n)}}{a(n)}s^2(\cos 1-1).\]
Putting \(d_n=n\sqrt{y} e^{-r(n)}/2a(n)\), by (\ref{rr}), we have \(d_n \rightarrow 1\), as \(n\rightarrow \infty\). Thus, choosing the positive number \(\epsilon = (1-\cos 1)/(3-2 \cos 1)\), there exists an integer \(n_1 (y)\), depending on \(y\) only, such that \(\vert d_n -1\vert<\epsilon\) for \(n\geq n_1 (y)\). Now, collecting these estimates, for some integer \(n_0 (y)\) we get
\begin{align*}\vert h_n (s)\vert&\leq \exp\left(-s^2 +d_n s^2 +2 d_n s^2 (\cos 1 -1)+1\right)\\
&\leq \exp\left(-s^2 + (1+\epsilon) s^2 +2 (1-\epsilon) s^2 (\cos 1 -1)+1\right)\\
&\leq \exp\left(- (1-\cos 1)s^2 +1\right),
\end{align*}
for \(n\geq n_0 (y)\) and all \(s\in \mathbb{R}\).
\item[ii)] Using (\ref{TaylorLem3}), (\ref{h}), (\ref{rr}) and keeping \(s\) fixed we get the estimate 
\begin{align*}\sum_{\nu=3}^{\infty}\vert b_{\nu}(n)\vert \left(\frac{\vert s\vert}{\sqrt{a(n)}}\right)^{\nu}&\leq\frac{n\sqrt{y}e^{-r(n)}}{a(n)^{3/2}}\sum_{\nu =3}^{\infty}\frac{1}{\nu!}\frac{\vert s \vert^{\nu}}{a(n)^{(\nu-3)/2}}\sum_{k=1}^{\infty}k^{\nu-1}\left(\sqrt{y}e^{-r(n)}\right)^{k-1}\\
&=\mathcal{O}\left(\frac{ne^{-r(n)}}{a(n)^{3/2}}\right)=\mathcal{O}\left(\frac{1}{\sqrt{\log n}}\right)
\end{align*}
from which the pointwise convergence in (\ref{hLem2}) follows.
\end{itemize} 
\end{proof}

Now we are prepared for the 

\begin{proof}[Proof of Theorem \ref{Main}] We start from the representation (\ref{I1}) and split the integral regarding the periodicity of the exponential \(e^{is}\) (observe the comments preceding formula (\ref{r}) and (\ref{Taylor}) above). Thus, we have
\begin{align}\nonumber&\sqrt{2a(n)\log y} \,e^{-\psi_n (0)} f_n\left(\tfrac{1}{y}\right)\\\nonumber
&=\sqrt{\frac{a(n)}{\pi}}\sum_{k=-\infty}^{\infty}\int\limits_{~2\pi k -1}^{2\pi k +1} e^{\psi_n (s)-\psi_n (0)} ds +\sqrt{\frac{a(n)}{\pi}}\sum_{k=-\infty}^{\infty}\int\limits_{2\pi k +1}^{2\pi (k+1) -1} e^{\psi_n (s)-\psi_n (0)} ds\\
& \label{HR}=H_n + R_n, \quad \text{say}. 
\end{align}
Next, we want to recognize \(R_n\) to be a null sequence. To this end, we apply (\ref{TaylorLem4}) to obtain the estimate
\begin{align*}\vert R_n \vert &\leq \sqrt{\frac{a(n)}{\pi}} \sum_{k=-\infty}^{\infty}\int\limits_{~2\pi k +1}^{2\pi (k +1)-1}\exp\left(\frac{-s^2}{2\log y}\right)\left\vert \frac{1+\sqrt{y}e^{-r(n)} e^{is}}{1+\sqrt{y}e^{-r(n)}}\right\vert^n ds\\
&\leq \sqrt{\frac{a(n)}{\pi}} \left(\frac{1+2\sqrt{y}e^{-r(n)} \cos 1 + y e^{-2r(n)}}{\left(1+\sqrt{y}e^{-r(n)}\right)^2}\right)^{n/2} \int\limits_{-\infty}^{\infty}e^{-s^2/2\log y}ds\\
&=\sqrt{a(n) 2 \log y}\exp\left(n \sqrt{y}e^{-r(n)} (\cos 1-1)+\mathcal{O}\left(\frac{\log^2 n}{n}\right)\right),
\end{align*}
as \(n \rightarrow \infty\). From (\ref{rr}) we get \(n\sqrt{y}e^{-r(n)}> \log n /2\log y\), if \(n\) is sufficiently large, and thus 
\begin{equation}\label{R} R_n = \mathcal{O}\left(\frac{\sqrt{\log n}}{n^{(1-\cos 1)/2\log y}}\right),
\end{equation}
as \(n \rightarrow \infty\).

To treat \(H_n\) we proceed as follows (see (\ref{TaylorLem4})):
\begin{align}\nonumber H_n &=\sqrt{\frac{a(n)}{\pi}}\sum_{k=-\infty}^{\infty}\int\limits_{-1}^{1} e^{\psi_n (t+2 \pi k)-\psi_n (0)} dt\\\nonumber
&=\sqrt{\frac{a(n)}{\pi}}\sum_{k=-\infty}^{\infty}\int\limits_{-1}^{1} \exp\left\{-\frac{(t+2\pi k)^2+2 i r(n) (t+2\pi k)}{2 \log y}\right\} \left(\frac{1+\sqrt{y}e^{-r(n)} e^{it}}{1+\sqrt{y}e^{-r(n)}}\right)^n dt\\\nonumber
&=\sum_{k=-\infty}^{\infty} \exp\left\{-\frac{2\pi^2 k^2 +2 \pi i k r(n)}{\log y}\right\}\\\nonumber
& \quad \quad \times \sqrt{\frac{a(n)}{\pi}} \int\limits_{-1}^{1} \exp\left\{-\frac{t^2+4\pi k t +2 i r(n)t}{2\log y}\right\}\left(\frac{1+\sqrt{y}e^{-r(n)} e^{it}}{1+\sqrt{y}e^{-r(n)}}\right)^n dt\\\nonumber
&=\sum_{k=-\infty}^{\infty} \exp\left\{-\frac{2\pi^2 k^2 +2 \pi i k r(n)}{\log y}\right\}\sqrt{\frac{a(n)}{\pi}} \int\limits_{-1}^{1} e^{-2\pi k t /\log y} \, e^{\psi_n(t)-\psi_n(0)}dt.
\end{align}
Now introducing
\[A_k (n)=\sqrt{\frac{a(n)}{\pi}} \int\limits_{-1}^{1} e^{-2\pi k t /\log y} \, e^{\psi_n(t)-\psi_n(0)}dt\]
we obtain (see (\ref{t1}))
\begin{align}\nonumber H_n &=\sum_{k=-\infty}^{\infty} \exp\left\{-\frac{2\pi^2 k^2 +2 \pi i k r(n)}{\log y}\right\}A_k (n)\\
\label{t3}&=\theta_3\left(\frac{\pi r(n)}{\log y}, e^{-2\pi^2/\log y}\right)+R_n^{\ast}
\end{align}
with 
\begin{equation}\label{t4}R_n^{\ast} =\sum_{k=-\infty}^{\infty} \exp\left\{-\frac{2\pi^2 k^2 +2 \pi i k r(n)}{\log y}\right\}\left(A_k (n)-1\right).
\end{equation}
In view of (\ref{HR})--(\ref{t4}) it is sufficient to prove that \(R_n^{\ast}\) tends to zero as \(n\rightarrow \infty\). To this end, in (\ref{t4}) we justify the interchange of summation and the limit as \(n\rightarrow \infty\). This indeed is possible because of the estimate (see (\ref{Taylor}), (\ref{TaylorLem4}), (\ref{h}))
\begin{align*}\vert A_k (n)\vert &\leq  e^{2 \pi\vert k\vert/\log y} \frac{1}{\sqrt{\pi}}\int\limits_{-\sqrt{a(n)}}^{\sqrt{a(n)}}e^{\Re\left(\psi_n\left(\frac{s}{\sqrt{a(n)}}\right)-\psi_n (0)\right)}ds\\
&=e^{2 \pi\vert k\vert/\log y} \frac{1}{\sqrt{\pi}}\int\limits_{-\infty}^{\infty} \vert h_n (s) \vert ds\\
&\leq e^{2 \pi\vert k\vert/\log y} \frac{e}{\sqrt{\pi}}\int\limits_{-\infty}^{\infty}e^{-(1-\cos 1)s^2}ds
\end{align*}
and Lebesgue's dominated convergence theorem combined with Lemma \ref{hLem}, i). Finally, we show that 
\begin{equation}\label{A} \lim_{n\rightarrow \infty} A_k(n)=1
\end{equation}
for any integer \(k \in \mathbb{Z}\). We have
\begin{align*}A_k(n)&=\sqrt{\frac{a(n)}{\pi}} \int\limits_{-1}^{1} e^{-2\pi k t /\log y} \, e^{\psi_n(t)-\psi_n(0)}dt\\
&=\frac{1}{\sqrt{\pi}} \int\limits_{-\sqrt{a(n)}}^{\sqrt{a(n)}}e^{-2\pi k s /\sqrt{a(n)}\log y} \, e^{\psi_n\left(s/\sqrt{a(n)}\right)-\psi_n(0)}ds\\
&=\frac{1}{\sqrt{\pi}} \int\limits_{-\infty}^{\infty}e^{-2\pi k s /\sqrt{a(n)}\log y}\, h_n (s)ds.
\end{align*}
Another application of Lebesgue's theorem (observe (\ref{hLem1})) in conjuction with (\ref{hLem2}) in Lemma \ref{hLem}, ii) implies (\ref{A}). Summarizing, we have proved 
\begin{equation}\label{S}f_n\left(\tfrac{1}{y}\right)=\frac{e^{\psi_n(0)}}{\sqrt{2a(n)\log y}}\left\{\theta_3\left(\frac{\pi r(n)}{\log y}, e^{-2\pi^2/\log y}\right)+o(1)\right\}.
\end{equation}
Finally, using (\ref{E2}), (\ref{r}), (\ref{TaylorLem1}) and (\ref{TaylorLem2}) we observe
\[2a(n)\log y =1+\frac{r(n)}{1+\sqrt{y}e^{-r(n)}}\sim  r(n)\]
and
\begin{align*}\psi_n(0)-\frac{r(n)^2+2r(n)}{2\log y}&= n \log\left(1+\sqrt{y}e^{-r(n)}\right)-\frac{r(n)}{\log y}\\
&=n\left(\log\left(1+\sqrt{y}e^{-r(n)}\right)-\frac{\sqrt{y}e^{-r(n)}}{1+\sqrt{y}e^{-r(n)}}\right)\\
&=\mathcal{O}\left(\frac{\log^2 n}{n}\right),
\end{align*}
as \(n \rightarrow \infty\).
Hence, in view of (\ref{S}), we arrive at (\ref{t2}) and the proof of Theorem \ref{Main} is complete.
\end{proof}

We conclude this section by comparing the asymptotic statements in Theorem \ref{Main} with the conjecture (\ref{R1}).

\begin{remark}\label{Re1} We briefly show that the factor in (\ref{R2}) is asymptotically equivalent to the quantity on the right-hand side of (\ref{R1}). This will follow from the relations
\begin{equation}\label{CR1}w(n)\sim r(n),\quad n\rightarrow \infty,
\end{equation}
\begin{equation}\label{CR2}w(n)-r(n)=o(1),\quad n\rightarrow \infty,
\end{equation}
\begin{equation}\label{CR3}w(n)^2-r(n)^2=o(1),\quad n\rightarrow \infty.
\end{equation}
Clearly, from (\ref{E1}) we have \(w(n)\sim \log n\) and thus (\ref{CR1}) results from (\ref{r}). Moreover, (\ref{E1}) and (\ref{r}) imply that 
\[e^{w(n)-r(n)}=\frac{n e^{-r(n)}\sqrt{y}\log y}{w(n)} \rightarrow 1,\quad n\rightarrow \infty,\]
and hence (\ref{CR2}) is established. Next, subtracting equation (\ref{E2}) from (\ref{E1}) gives 
\[w(n)e^{w(n)}-r(n)e^{r(n)}=r(n)\sqrt{y},\]
or equivalently
\[w(n)\left(e^{w(n)-r(n)}-1\right)=e^{-r(n)}r(n)\sqrt{y}- (w(n)-r(n)).\]
Hence, the fact that the right-hand side tends to zero as \(n\rightarrow \infty\) (use (\ref{r}) and (\ref{CR2})) implies that
\[w(n)(w(n)-r(n))\rightarrow 0,\quad n\rightarrow \infty.\]
Furthermore, we observe
\[0\leq w(n)^2-r(n)^2 =(w(n)-r(n))(w(n)+r(n))\leq 2 w(n)(w(n)-r(n))\]
from which (\ref{CR3}) follows.
\end{remark}
\begin{remark} According to the comments given in Remark \ref{Re1} the only significant difference between the approximations (\ref{R1}) and (\ref{t2}) consists in the presence of the Fourier series 
\[\rho_n(y)=2\sum_{k=1}^{\infty} e^{-2\pi^2 k^2/\log y} \cos\left(2\pi k r(n)/\log y\right)\]
in (\ref{R3}), which obviously does not tend to zero, as \(n \rightarrow \infty\). However, the trivial estimate 
\[\vert \rho_n(y)\vert \leq \frac{2}{e^{2\pi^2/\log y}-1}\]
explains why numerical calculations do not suggest that (\ref{R1}) could not be true. For instance
\[\vert \rho_n(2)\vert \leq 10^{-12}\]
holds.
\end{remark}

\section{Monotonicity properties}

In this final section we add an elementary property of the sequence \(\left(f_n\left(\tfrac{1}{y}\right)\right)_0^{\infty}\) (see \cite{Feller}, Ch. VII, 1-3).

\begin{definition} A sequence of real numbers \((a_n)_0^{\infty}\) is called absolutely monotonic, if for any \(r, n \in \mathbb{N}\cup\{0\}\)
\[\Delta^r a_n \geq 0,\]
where \(\Delta a_n =a_{n+1}-a_n\) and \(\Delta^{0} a_n= a_n\), \(\Delta^{r+1}=\Delta \Delta^r\).
\end{definition}
In particular, absolutely monotonic sequences are increasing (\(r=1\)) and convex (\(r=2\)).

\begin{theorem} If \(y>1\), then the sequence \(\left(f_n\left(\tfrac{1}{y}\right)\right)_0^{\infty}\) is absolutely monotonic. 
\end{theorem}
\begin{proof} From Lemma \ref{IntRep} and formula (\ref{1}) we get
\begin{align*} \Delta^r f_n\left(\tfrac{1}{y}\right) &=\frac{1}{\sqrt{2\pi \log y}}\int\limits_{-\infty}^{\infty}\exp\left\{-\frac{s^2}{2 \log y}\right\} \left(1+\sqrt{y} e^{is}\right)^n \left(\sqrt{y}e^{is}\right)^r ds\\
&=\sum_{k=0}^{n}\binom{n}{k}\frac{1}{y^{\binom{k+r}{2}}},
\end{align*}
which is positive.
\end{proof}

\section*{Acknowledgement} The authors would like to express their gratitude to the anonymous referees who pointed out an error in an earlier draft of the paper. This finally led to the present version.

\end{document}